\documentclass[10pt]{article}
\usepackage{amsmath}
\usepackage{amssymb}
\usepackage{euscript}
\usepackage[latin1]{inputenc}

\normalbaselines
\newcommand{\qed}{\sqcap\kern-6.8pt\sqcup}

\newenvironment{proof}
      {\par\noindent{\it Proof\/: }\nopagebreak\normalsize}%
                                  {\linebreak[2]\hspace*{\fill}$\qed$\ifdim\lastskip<10pt
       \removelastskip \penalty-200  \vskip10pt  \fi}

\font\frten=eufm10 at 10pt
\font\freight=eufm10
\font\frsix=eufm8
\newfam\frfam\textfont\frfam=\frten
\scriptfont\frfam=\freight
\scriptscriptfont\frfam=\frsix

\newcommand{\CC}{\mathbb{C}}
\newcommand{\PP}{\mathbb{P}}
\newcommand{\RR}{\mathbb{R}}
\newcommand{\NN}{\mathbb{N}}
\newcommand{\ZZ}{\mathbb{Z}}
\newcommand{\QQ}{\mathbb{Q}}

\newtheorem{thm}{Theorem}[section]
\newtheorem{prop}[thm]{Proposition}
\newtheorem{cor}[thm]{Corollary}
\newtheorem{lem}[thm]{Lemma}
\newtheorem{defn}[thm]{Definition}

\newtheorem{rema}[thm]{Remark}

\def \dim{{\rm dim}}

\def \L{{\cal O}}

\def \Div{{\rm Div}}
\def \Gal{{\rm Gal}}

\def \deg{{\rm deg}}
\def \ind{{\rm ind}}
\def \coind{{\rm coind}}

\begin{document}

\title{Very special divisors on $4$-gonal real algebraic curves}

\author{Jean-Philippe Monnier\\
       {\small D\'epartement de Math\'ematiques, Universit\'e d'Angers,}\\
{\small 2, Bd. Lavoisier, 49045 Angers cedex 01, France}\\
{\small e-mail: monnier@tonton.univ-angers.fr}}
\date{}
\maketitle

\begin{abstract}
{Given a real curve, we study special linear 
systems called ``very special''
for which the dimension does not satisfy a
 Clifford type inequality. We classify all these very special
 linear systems when the gonality of the curve is small.}
\end{abstract}

\section{Introduction and preliminaries}

In this note, a real algebraic curve $X$ is a smooth proper geometrically
integral scheme over $\RR$ of dimension $1$.
A closed point $P$ of $X$ 
will be called a real point if the residue field at $P$ is $\RR$, and 
a non-real point if the residue field at $P$ is $\CC$.
The set of real points $X(\RR)$ of $X$ decomposes into 
finitely many connected components, whose number will be denoted by $s$.
By Harnack's Theorem (\cite[Th. 11.6.2 p. 245]{BCR}) we know that $s\leq g+1$,
where $g$ is the genus of $X$. 
A curve with $g+1-k$ real connected components
is called an $(M-k)$-curve. Another topological invariant associated to
$X$ is $a(X)$, the number of connected components of $X(\CC)\setminus X(\RR)$
counted modulo $2$. The pair $(s,a(X))$ is called the topological type
of $X$. If $a(X)=0$ then $s=g+1\,\mod 2$ (see \cite{K}) and $X$ is
called a separating curve.

We will denote by $X_{\CC}$ the base extension of $X$
to $\CC$. 
The group $\Div(X_{\CC})$ of divisors on $X_{\CC}$
is the free abelian group on the 
closed points of $X_{\CC}$. The Galois group $\Gal(\CC /\RR)$
acts on the complex variety $X_{\CC}$ and also on $\Div(X_{\CC})$. 
We will always indicate this action by a bar. Identifying $\Div(X)$
and $\Div(X_{\CC})^{\Gal(\CC /\RR)}$,
if $P$ is a non-real point of $X$ then 
$P=Q+\bar{Q}$ 
with $Q$ a closed point of $X_{\CC}$.
The group $\Div(X)$ of divisors on $X$ 
is then the free abelian group generated by the 
closed points of $X$. If $D$ is a divisor on $X$, we will denote by
$\L (D)$
its associated invertible sheaf. The dimension of the space 
of global sections of this sheaf will be denoted
by $h^0 (D)$. 
Since a principal divisor has an even degree on each connected 
component of $X(\RR)$ (e.g. \cite{G-H} Lem. 4.1),
the number $\delta(D)$ (resp. $\beta(D)$) of connected components $C$
of $X(\RR )$
such that the degree of the restriction of $D$ to $C$ is odd (resp even) 
is an invariant of the linear system $|D|$
associated to $D$. If $h^0 (D)>0$, the dimension of the linear system
$|D|$ is $\dim\, |D|=h^0 (D)-1$.
Let $K$ be the canonical divisor. If 
$h^0 (K-D)=\dim\, H^1 (X,\L (D) )>0$, $D$ is said to be special. If not,
$D$ is said to be non-special. By Riemann-Roch, if $\deg(D)>2g-2$ then $D$
is non-special. Assume $D$ is effective of degree $d$. If
$D$ is non-special then the dimension of the linear system $|D|$ is given by 
Riemann-Roch. If $D$ is special, then the dimension of the linear system
$|D|$ satisfies 
$$\dim\, |D|\leq \frac{1}{2} d.$$
This is the well known Clifford inequality for complex curves that
works also for real curves. The reader is referred to \cite{ACGH} and
\cite{Ha} for more details on special divisors. Concerning real curves,
the reader may consult \cite{G-H}.

Huisman (\cite[Th. 3.2]{Hu})
has shown that: 
\begin{thm} \label{Huisman}
Assume $X$ is an $M$-curve (i.e. $s=g+1$) or an $(M-1)$-curve (i.e. $s=g$).
Let $D\in \Div(X)$ be an effective and special divisor of degree $d$.
Then $$\dim\, |D|\leq \frac{1}{2} (d-\delta (D)).$$
\end{thm}

Huisman inequality
is not valid for all real curves and we have the
following theorem.
\begin{thm} \cite[Th. A]{Mo1}
\label{cliffreelA}
Let $D$ be an effective and special divisor of degree $d$. Then either
$$\dim\, |D|\leq \frac{1}{2} (d-\delta (D))\eqno{\rm (Clif 1)}$$
or 
$$\dim\, |D|\leq \frac{1}{2} (d-\beta (D))\eqno{\rm (Clif 2)}$$
Moreover, $D$ satisfies the first inequality
if either $s\leq 1$ or $s\geq g$.
\end{thm}

In this note we are interested in special divisors that do not satisfy
the inequality ${\rm (Clif 1)}$ given by Huisman.
\begin{defn} Let $D$ be an effective and special divisor of degree
  $d$. We say that $D$ is a very special divisor (or $|D|$ is a very
  special linear system)
  if $D$ does not
  satisfy the inequality (Clif 1) i.e. $\dim\, |D|>\frac{1}{2}
  (d-\delta (D))$. If $D$ is very special then there exists $k\in\NN$
  such that $$\dim\, |D|= \frac{1}{2} (d-\delta (D))+k+1$$ and $k$ is
  called the index of $D$ denoted by $\ind(D)$.
\end{defn}

We can reformulate Theorem \ref{cliffreelA} with the concept of very
special divisors.
\begin{thm} \cite[Th. B]{Mo1}
\label{cliffreelB}
Let $D$ be an effective and very special divisor of degree $d$.
Then $$\dim\, |D|\leq \frac{1}{2} (d-\frac{1}{2}(s-2))\eqno{(*)}.$$
\end{thm}

In the previous cited paper, a result is obtained in this direction.
\begin{thm} \cite[Th. 2.18]{Mo1}
\label{siegalite}
Let $D$ be a very special and effective divisor of degree $d$ 
on a real curve $X$ such that (*) is an equality i.e.
$$r=\dim\, |D|= \frac{1}{2} (d-\frac{1}{2}(s-2))$$
then $X$ is an hyperelliptic curve with 
$\delta( g_2^1 )=2$ and $|D|=rg_2^1 $ with $r$ odd.
\end{thm}

Let $D\in \Div(X)$ be a divisor with the property that $\L (D)$ has at
least one nonzero global section. The linear system $|D|$ is called
base point free if $h^0(D-P)\not=h^0(D)$ for all closed points
$P$ of $X$. If not, we may write $|D|=E+|D'|$ with $E$ a non zero
effective divisor called the base divisor of $|D|$, and with $|D'|$
base point free. A closed point
$P$ of $X$ is called a base point of $|D|$ if $P$ belongs to the support of
the base divisor of $|D|$. We note that $$\dim\, |D|=\, \dim\, |D' |.$$

As usual, a $g_d^r$ is an $r$-dimensional complete linear system
of degree $d$ on $X$. Let $|D|$ be a base point free $g_d^r$ on $X$.
The linear system $|D|$ defines a morphism $\varphi_{|D|} :\, X\rightarrow
\PP_{\RR}^r$ onto a non-degenerate (but maybe singular) curve in
$\PP_{\RR}^r$ i.e. $\varphi_{|D|}(X)$ is not contained in any hyperplane of
$\PP_{\RR}^r$. If $\varphi_{|D|}$ is birational (resp. an isomorphism) onto
$\varphi_{|D|} (X)$, the $g_d^r$ (or $D$) is called simple (resp. very
ample). Let $X'$ be the normalization of $\varphi_{|D|}(X)$, 
and assume $D$ is not
simple i.e. $|D-P|$ has a base point for any closed point $P$ of $X$. Thus,
the induced morphism 
$\varphi_{|D|} :X\rightarrow X'$ is a non-trivial covering
map of degree $t\geq 2$. In particular, there exists $D'\in \Div(X')$ such that
$|D'|$ is a $g_{\frac{d}{t}}^{r}$ and such that $D=\varphi_{|D|}^* (D')$, i.e.
$|D|$ is induced by $X'$. If $g'$ denote the genus of $X'$,
$|D|$ is classically called compounded of an involution of order $t$
and genus $g'$. In the case $g'>0$, we speak of an irrational involution 
on $X$.

Concerning non-simple very special divisors, 
a complete description is given in \cite{Mo2}:
\begin{thm} \cite[Thm. 2.5, Thm. 4.1]{Mo2}
\label{nonsimple}
Let $D$ be a very special divisor of degree $d$ such that the base
point free part of $|D|$ is non-simple. Then 
\begin{description}
\item[({\cal{i}})] $D$ is base point free,
\item[({\cal{ii}})] $\delta(D)=s$,
\item[({\cal{iii}})] the index of $D$ is null.
\end{description}
If moreover $\dim\, |D|=r>1$ then the morphism $\varphi_{|D|}
:X\rightarrow X'$ is a non-trivial covering
map of degree $2$ and $D=\varphi_{|D|}^*
(D')$ with $D'\in \Div(X')$ such that $|D'|=g_{\frac{d}{2}}^{r}$. Let $g'$ denote the genus of $X'$ and let $s'$ be the number of
connected components of $X'(\RR)$, we have the following additional 
properties: 
\begin{description}
\item[({\cal{iv}})] $D'$ is a base point free non-special divisor and
  $X'$ is an $M$-curve.
\item[({\cal{v}})] $s$ is even, $s'=\frac{s}{2}$, $r$ is odd and
  $\delta(D')=s'$. 
\item[({\cal{vi}})] $a(X)=0$ and $g$ is odd and
there is a very special pencil on $X$.
\end{description}
\end{thm}

In this note, we give conditions under which a
real curve with few real connected
components can have a very special system.
\begin{thm} 
\label{introthm1}
Let $X$ be real curve with $s\leq 4$. If $X$ has a
  very special linear system then $X_{\CC}$ and $X$ are $s$-gonal,
  $s\geq 2$,
  $X$ has a very special pencil and $X$ is a separating curve i.e. $a(X)=0$.
\end{thm} 

If $s\geq g-4$, we prove that the existence of a very special linear
system implies also the existence of a very special pencil.

By the previous theorem, the existence of a very special linear system
on a real curve, with $s\leq 4$, forces the gonality of the curve to
be $s$. The following theorem concerns very special linear series on a
real curve with a small gonality.
\begin{thm} 
\label{introthm2}
Let $X$ be real curve such that $X$ and $X_{\CC}$ are both $n$-gonal
with $2\leq n\leq 4$. If $X$ has a
very special linear system then 
$X$ has a very special pencil and $X$ is a separating curve
i.e. $a(X)=0$.
Moreover, if $|D|$ is a very special linear system then $\ind(D)=0$,
$\delta(D)=s$, $|D|$ and $|K-D|$ are base point free.
\end{thm} 

In the last section of this note, we improve the result of Theorem
\ref{cliffreelB}. 
\begin{thm}
\label{introthm3}
Let $|D|$ be a very special linear system of degree $d$ on a real curve $X$.
\begin{description}
\item[({\cal{i}})] We have $$\dim |D|\leq
  \frac{1}{2}(d-\frac{s-2}{2}),$$
with equality if and only if $X$ is
hyperelliptic, the $g_2^1$ is very special and $s=2$.
\item[({\cal{ii}})] Assume $X$ is not hyperelliptic. We have $$\dim |D|\leq
  \frac{1}{2}(d-\frac{s-1}{2}),$$
with equality if and only if $X$ is
trigonal, a $g_3^1$ is very special and $s=3$.
\item[({\cal{iii}})] Assume $X$ is not hyperelliptic and not
  trigonal. 
  We have $$\dim |D|\leq
  \frac{1}{2}(d-\frac{s}{2}),$$
with equality if and only if $X$ is
4-gonal, a $g_4^1$ is very special and $s=4$.
\item[({\cal{iv}})] Assume $X$ has gonality $\geq 5$. We have $$\dim |D|\leq
  \frac{1}{2}(d-\frac{s+1}{2}).$$
\end{description}
\end{thm}

\section{Properties of very special divisors}

In this section, we recall and establish some results on very special
linear systems.

Let $D$ be a special and effective divisor. The linear system $|D|$ is
called primitive if $|D|$ and $|K-D|$ are base point free.
If $|D|$ is base point free and $F$ is the base divisor of $|K-D|$
then $|D+F|$ is primitive (it is called the primitive hull of $|D|$)
and satisfies $\dim |D+F|=\dim |D|+\deg(F)$.

We recall that if $D\in Div(X)$ then $\delta (D)=\delta(K-D)$.
By the previous remark and Riemann-Roch, we get:
\begin{lem} \cite[Lem. 2.4]{Mo2}
\label{residuel}
Let $D$ be a very special divisor then $K-D$ is also very special
and $\ind(D)=\ind(K-D)$.
\end{lem}

\begin{lem} 
\label{basepointfreepart}
Let $D$ be an effective divisor. Let $F$ be the base divisor of $|D|$.
If $D$ is very special then the base point free part $|D-F|$ of $|D|$
is also very special and $\ind(D-F)\geq \ind(D)$.
Moreover $\ind(D-F)= \ind(D)$ if and only if $F=P_1+\ldots +P_f$ with
the $P_i$ some real points among the $\delta (D)$ real connected
components on which the degree of the restriction of $D$ is odd, such
that no two of them belong to the same real connected component.
\end{lem}

\begin{proof}
Set $d=\deg (D)$ and $k=\ind (D)$.

Suppose that there exists a non-real point $Q+\bar{Q}$ such that
$Q+\bar{Q}\leq F$. Then $\dim |D|=\dim
|D-Q-\bar{Q}|=\frac{1}{2}(d-\delta(D))+k+1=\frac{1}{2}((d-2)-\delta(D))+(k+1)
+1$ and $\ind(D-Q-\bar{Q})=\ind (D)+1$. 

Suupose there are two real points $P,P'$ belonging to the same real connected
component, such that $P+P'\leq F$, then as before, $\ind(D-P-P')=\ind (D)+1$.

Suppose that there exists a real point $P$ belonging to a connected component on which
the degree of the restriction of $D$ is even, such that $P\leq F$. 
Then $\dim |D|=\dim
|D-P|=\frac{1}{2}(d-\delta(D))+k+1=\frac{1}{2}((d-1)-(\delta(D)+1))+(k+1)
+1=\frac{1}{2}(\deg(D-P)-(\delta(D-P))+(k+1)
+1$ and $\ind(D-P)=\ind (D)+1$.

Suppose that there exists a real point $P$ belonging to a connected 
component on which
the degree of the restriction of $D$ is odd, is contained in the base divisor
of $|D|$. Then $\dim |D|=\dim
|D-P|=\frac{1}{2}(d-\delta(D))+k+1=\frac{1}{2}((d-1)-(\delta(D)-1))+k
+1=\frac{1}{2}(\deg(D-P)-(\delta(D-P))+k
+1$ and $\ind(D-P)=\ind (D)$.
\end{proof}

\begin{cor}
\label{primitivehull}
Let $D$ be very special and assume $|D|$ is base point free then the
primitive hull $|D'|$ of $|D|$ is very special and $\ind(D')\geq \ind
(D)$. Let $E\in |D'-D|$, then $\ind(D')= \ind(D)$ if and only if 
$E=P_1+\ldots +P_e$ with
the $P_i$ some real points among the $\delta (D)$ real connected
components on which the degree of the restriction of $D$ is odd, such
that no two of them belong to the same real connected component.
\end{cor}

\begin{proof}
Let $E$ denote the base divisor of $K-D$. Then $K-D$ is very special
of index $\ind(D)$ (Lemma \ref{residuel}) and $K-D-E$ is also very
special of index $\geq\ind(D)$ by Lemma \ref{basepointfreepart}. By
Lemma \ref{residuel}, $D'=D+E=K-(K-D-E)$ is very
special of index $\geq\ind(D)$.

Assume $\ind(D')= \ind(D)$ then $\ind (K-D)=\ind (K-D-E)$ and by Lemma
\ref{basepointfreepart} we are done.
\end{proof}

\begin{thm} \cite[Thm. 3.6]{Mo2}
\label{dim2}
If $D$ is a very special divisor then $\dim |D|\not= 2$.
\end{thm}

\begin{prop}
\label{veryspecialpencil}
Let $D$ be a very special divisor of degree $d$ such that 
$\dim\, |D| =1$. Then $D=P_1 +\dots + P_{s}$ with $P_1 ,\ldots ,P_{s} $
some real points of $X$ 
such that no two of them belong to the same connected component
of $X(\RR )$ i.e. $d=\delta (D)=s$ and $\ind(D)=0$. Moreover $D$ is primitive.
\end{prop}

\begin{proof}
By \cite[Prop. 2.1]{Mo2}, we only have to prove $K-D$ is base point
free. 

Suppose that $P$ is a real base point of $|K-D|$. Then 
$\dim |D+P|=2=
\frac{1}{2}((s+1)-(s-1))
+1=\frac{1}{2}(\deg(D+P)-(\delta(D+P))
+1$ and $D+P$ is very special, impossible by Theorem \ref{dim2}.

Suppose that $Q+\bar{Q}$ 
is a non-real base point
of $|K-D|$. Then for a general choice of a real point $P$, we get 
$\dim |D+Q+\bar{Q}-P|=2=\frac{1}{2}((s+1)-(s-1))
+1=\frac{1}{2}(\deg(D+Q+\bar{Q}-P)-(\delta(D+Q+\bar{Q}-P))
+1$ and $D+Q+\bar{Q}-P$ is very special, impossible by Theorem
\ref{dim2}.
\end{proof}

\begin{prop}
\label{construction}
Let $D$ be a
very special divisor of degree $d$. Let $k=\ind (D)$.
Choose $k$ distinct real connected components $C_1,\ldots,C_k$ on
which the degree of the restriction of $D$ is odd. For $i=1,\ldots,k$, take
$(P_i,R_i)\in C_i\times C_i$. Take a real point $Q_j$ in each real
connected component on
which the degree of the restriction of $D$ is even. Then 
$$D'=D-\sum_{i=1}^k (P_i+R_i)-\sum_{j=1}^{\beta (D)}Q_j$$
is very special. Moreover, $\dim |D'|=\frac{1}{2}(d+\delta(D))-k-s+1$
and $\ind (D')=0$ if the points $P_i,R_i,Q_j$
are sufficiently general.
\end{prop}

\begin{proof}
Assume $\dim |D|=r=\frac{1}{2}(d-\delta(D))+k+1$. We have $r\geq k+1$
since $D$ can be chosen effective.

We have $k\leq \frac{\delta(D)}{2}$ by Clifford Theorem.

Let $D_1=D-\sum_{i=1}^k P_i$. Then $\delta(D_1)=\delta (D)-k$ and 
$\dim |D_1|\geq r-k=\frac{1}{2}(d-\delta(D))+1-k=\frac{1}{2}(\deg
(D_1)-\delta(D_1))+1\geq 1$. We see that $D_1$ is very special and
$\ind(D_1)=0$ if the points $P_i$ are general.

If $\beta (D_1)=0$, the proof is done since $D_1$ is very special.
Assume $\beta(D_1)>0$. Let $P$ be a real point such that 
$P$ belongs to a connected component on
which the degree of the restriction of $D_1$ is even. Then 
$\dim |D_1-P|\geq \dim |D_1|-1\geq \frac{1}{2}(\deg
(D_1)-\delta(D_1))=\frac{1}{2}(\deg
(D_1-P)-\delta(D_1-P))+1$. Since $h^0 (D_1-P)\geq\dim |D_1| >0$ then
$D_1-P$ is linearly equivalent to an effective divisor and
therefore $\deg
(D_1-P)\geq \delta(D_1-P)$ and $\dim |D_1-P|\geq 1$ and $D_1-P$ is
very special. 
In the case $P$ is general, we get $\ind (D_1-P)=\ind (D_1)$.
By repeating the same reasoning for $D_1-P$, we prove the
proposition.
\end{proof}

The following lemma was proved differently in \cite[Lem. 3.3]{Mo2}.
\begin{lem}
\label{condition1}
Let $D$ be a very special divisor of degree $d$ and index $k$. Then 
$$d+\delta (D)\geq 2s+2k.$$
\end{lem}

\begin{proof}
By Proposition \ref{construction}, the divisor $D'=D-\sum_{i=1}^k
(P_i+R_i)-\sum_{j=1}^{\beta (D)}Q_j$ is very special and therefore 
$\dim |D'|\geq 1$. Choosing the points $P_i,R_i,Q_j$ sufficiently
general, we have $\dim |D|=2k+\beta (D)+\dim |D'|\geq
2k+s-\delta(D)+1$. We obtain 
$$\frac{1}{2}(d-\delta(D))+k+1\geq 2k+s-\delta(D)+1$$
i.e.$$d+\delta (D)\geq 2s+2k.$$
\end{proof}

\begin{cor}
\label{egalitecondition1}
Let $D$ be a very special divisor of degree $d$ and index $k$.
We have $d+\delta (D)= 2s+2k$ if and only if
$|D|$ is a pencil.
\end{cor}

\begin{proof}
Assume $d+\delta (D)= 2s+2k$. The linear system $|D'|$ of Proposition
\ref{construction} is a very special pencil. Hence $|D'|$ is primitive
(Proposition \ref{veryspecialpencil})
and therefore $D'=D$.

The converse follows easily from Proposition \ref{veryspecialpencil}.
\end{proof}

We improve the result of Lemma \ref{condition1}.
\begin{prop}
\label{condition2}
Let $D$ be a very special divisor of degree $d$ and index $k$ such
that $|D|$ is not a pencil. Then 
$$d+\delta (D)\geq 2s+2k+4.$$
The very special divisor $D'$ constructed from $D$ in
Proposition \ref{construction} satisfies $$\dim |D'|\geq 3.$$
Moreover $d+\delta (D)= 2s+2k+4$ if and only if $\dim |D'|= 3.$
\end{prop}

\begin{proof}
Choosing the points $P_i,R_i,Q_j$ sufficiently
general in Proposition \ref{construction}, the linear system $|D'|$ 
is very special and $\dim |D'|=\frac{1}{2}(d+\delta(D))-k-s+1$.
Note that $|D'|$ is not primitive if $\delta(D)<s$ or $k>0$. 
By Lemma \ref{condition2}, we have 
$d+\delta (D)\geq 2s+2k+2$. If $d+\delta (D)= 2s+2k+2$ then 
$\dim |D'|=2$, impossible by Theorem \ref{dim2}.
Therefore $\dim |D'|\geq 3$ and 
we have $\dim |D|=2k+\beta (D)+\dim |D'|\geq
2k+s-\delta(D)+3$. We obtain 
$$\frac{1}{2}(d-\delta(D))+k+1\geq 2k+s-\delta(D)+3$$
i.e.$$d+\delta (D)\geq 2s+2k+4.$$
\end{proof}

\begin{prop}
\label{M-2} Let $X$ be an $(M-2)$-curve or an $(M-4)$-curve. 
If $X$ has a 
very special linear system
then it is a very special pencil $g_s^1$ or its residual $K-g_s^1$.
If $s=g-1$ then the residual of a very special pencil is also a very
special pencil.
\end{prop}

\begin{proof}
Let $D$ be a very special divisor of degree $d$ and index $k$ on $X$.
By Lemma \ref{residuel}, we may assume $d\leq g-1$ and then 
$d+\delta (D)\leq 2g-2$ (resp. $\leq 2g-4$) if $s=g-1$ (resp. $s=g-3$).
If $|D|$ is not a pencil then $d+\delta (D)\geq 2g-2+2k+4$
(resp. $\geq 2g-6+2k+4$) if $s=g-1$ (resp. $s=g-3$),
impossible. Hence $|D|$ is a very special pencil and by Riemann-Roch,
$K-D$ is also a very special pencil in the case $s=g-1$.
\end{proof}

\begin{rema} The very special pencils give separating morphisms
  $X\rightarrow \PP^1$ in the sense of Coppens \cite{Co}.
\end{rema}

\begin{prop}
\label{M-3} Let $X$ be an $(M-3)$-curve or an $(M-5)$-curve. 
Then $X$ has no very special linear systems.
\end{prop}

\begin{proof}
Let $D$ be a very special divisor of degree $d$ and index $k$ on $X$.
By Lemma \ref{residuel}, we may assume $d\leq g-1$ and then 
$d+\delta (D)\leq 2g-3$ (resp. $\leq 2g-5$) if $s=g-2$ (resp. $s=g-4$).
If $|D|$ is not a pencil then $d+\delta (D)\geq 2g-4+2k+4$
(resp. $\geq 2g-8+2k+4$) if $s=g-2$ (resp. $s=g-4$),
impossible. Hence $|D|$ is a very special pencil and then $a(X)=0$,
impossible.
\end{proof}

From the above results we get:
\begin{thm}
\label{s>=g-4}
Let $X$ be a real curve such that $s\geq g-4$. If $X$ has a very
special linear system then $X$ has a very special pencil.
\end{thm}

\section{Very special webs}

\begin{prop}
\label{dim3}
Let $D$ be a very special divisor of degree $d$ and index $k$ such
that $\dim |D|=3$. Then $|D|$ is base point free, $d=s+4$, $k=0$ and 
$\delta (D)=s$.  
\end{prop}

\begin{proof}
We have $D'=D$ in Proposition \ref{construction} and thus
$d+\delta (D)= 2s+2k+4$.
Since $D=D'$ then we have $k=0$ and 
$\delta (D)=s$. Therefore, we obtain $d=s+4$. 
By Proposition \ref{basepointfreepart}, the base point free part of $|D|$ 
is also very special of dimension $3$, hence its degree is $s+4$ and
thus $D$ is base point free.
\end{proof}

The following proposition is important in the sequel. We give new 
examples of very special simple linear systems. This proposition was
inspired by M. Coppens.
\begin{prop}
\label{quadrique}
Let $|D|$ be a very special simple linear system of degree $d$ 
such that $\dim |D|=3$ and 
$X'=\varphi_{|D|}(X)\subset \PP^3$ is contained in an 
irreducible real quadric surface $Q$.
Then 
\begin{description}
\item[({\cal{i}})] The rank of $Q$ is $4$ and 
$Q\simeq \PP_{\RR}^1\times \PP_{\RR}^1$.
\item[({\cal{ii}})] $X'$ is a curve of bi-degree $(s,4)$ on $Q$.
\item[({\cal{iii}})] $|D|=g_s^1+g_4^1$, $g_s^1$ and $g_4^1$ 
are the pull-backs of 
the linear pencils on $X'$ cut out by the rulings of $Q$.
\item[({\cal{iv}})] $g_s^1$ is a very special pencil and $\delta
  (g_4^1)=0$. 
\item[({\cal{v}})] $s\geq 3$ and in the case $s=3$ 
  then $X'$ is smooth and $|D|=K-g_3^1$.
\item[({\cal{vi}})] If $s=4$ and $d\leq g-1$ then $X'$ is smooth.
\end{description}
\end{prop}

\begin{proof}
By Proposition \ref{dim3}, the degree $d$ of $D$ is $s+4$,
$\delta(D)=s$ and 
$\ind (D)=0$.

If the rank of $Q$ is $3$, then $|D|=|2F|$ where $|F|$ 
is the pencil induced by the ruling of $Q$. 
This case is not possible since $\delta (D)\not= 0$ by Clifford Theorem.

Thus $Q$ is smooth. Assume $Q(\RR)\simeq S^2$. 
The rulings of $Q$ induce complex and conjugated 
pencils $|F|$ and $|\bar{F}|$ on $X$ such that 
$|D|=|F+\bar{F}|$; this is again impossible since $\delta (D)\not= 0$.

We have $Q\simeq \PP_{\RR}^1\times \PP_{\RR}^1$ and $|D|=g_a^1+g_b^1$ 
where $g_a^1$ and $g_b^1$ are induced by the real rulings of $Q$.
The hyperplane section $H$ giving the embedding 
$Q\hookrightarrow \PP^3$ is of bi-degree $(1,1)$
on $Q$
and then $X'=\varphi_{|D|}(X)\subset \PP^3$ is a curve of bi-degree
$(a,b)$ on $Q$. Moreover, $X$ and $X'$ are birational since $|D|$ is simple.
We get $$a+b=s+4.$$
We have $H_1(Q(\RR),\ZZ/2)=\ZZ/2\times\ZZ/2$ and the possible 
types of the image of the connected components of $X(\RR)$ are 
$(0,0)$, $(1,0)$,
$(0,1)$ and $(1,1)$. Let $a'$, $b'$ and $c'$ be respectively the 
number of connected components of type $(1,0)$, $(0,1)$ and $(1,1)$.
We have $a'+b'=\delta (D)=s$, $a'+c'=\delta (g_a^1)$ and $b'+c'=\delta (g_b^1)$.

Suppose $a'\leq a-2$ and $b'\leq b-2$. Since $a'+b'=s$ and $a+b=s+4$,
we get $a'= a-2$ and $b'= b-2$. Since a connected component of
type $(1,0)$ intersect a connected component of type $(0,1)$, if $g$
denote the genus of $X$, we get $$g\leq ab-a-b+1-(a-2)(b-2)$$
i.e. $$g\leq a+b-3=s+1,$$
which is impossible by Theorem \ref{Huisman} and Proposition \ref{M-2}.

So we can assume $a'=a$ i.e. $g_a^1$ is a very special pencil. We have
$a=s$ and it follows that $b=4$ and $b'=c'=0$. If $s\leq 2$ then $g\leq
3$ by the genus formula and this is again impossible by the
propositions \ref{M-2} and \ref{M-3}. 

Assume $s=3$. Let $\mu$ denote the multiplicity of the singular
locus of $X'$. We have $g=6-\mu$ by the genus formula.
Since $D$ is special, we have $\dim |D|=3>d-g=1+\mu$ i.e. $\mu\leq 1$.
If $\mu =1$, then $s$ and $g$ have the same parity, impossible since
$a(X)=0$ ($X$ has a very special pencil). Thus $X'$ is smooth and $X$ is
an $(M-4)$-curve. 
By Proposition \ref{M-2}, $|D|$ is residual to a very special $g_3^1$.
Since $X$ has a simple very special linear system, $X$ cannot be
hyperelliptic and $X$ is trigonal with a unique $g_3^1$ ($g>4$)
such that $|D|=g_3^1+g_4^1$.

Assume now $s=4$ and $d=s+4=8\leq g-1$. 
Let $\mu$ denote the multiplicity of the singular
locus of $X'$. We have $g=9-\mu$ by the genus formula and it follows
that $\mu=0$.
\end{proof}

We study the converse of the previous proposition.
\begin{prop}
\label{quadrique2}
Let $X$ be a smooth curve of bidegree $(s,4)$ on a hyperbolo\"{\i}d
$Q\subset \PP_{\RR}^3$ with $s\geq 3$ denoting the number of connected
components of $X(\RR)$ and such that all the connected components of
$X(\RR)$ are of type $(1,0)$. Then the embedding $X\hookrightarrow
\PP^3$ is given by a simple very special linear system
$|D|=g_s^1+g_4^1$. 
\end{prop}

\begin{proof}
The embedding $X\hookrightarrow
\PP^3$ is clearly given by a simple linear system
$|D|=g_s^1+g_4^1$ with $\delta (g_s^1)=s$ and $\delta (g_4^1)=0$.
It remains to show that $D$ is special. Since $X$ is smooth, we have
$g=3s-3$ and thus $D$ is special if $\dim |D|=3>d-g=s+4-3s+3$ i.e. $s>2$.
\end{proof}

\begin{rema} The existence of curves of bidegree $(s,4)$ with
  prescribed types of 
  Proposition \ref{quadrique2} for the real connected
  components is proved by Zvonilov \cite{Zv} for $s=3$ and $s=4$.
\end{rema}

\section{Curves with a small number of real connected components}

We study the existence of very special linear systems on real curves with
$s\leq 4$.

The following theorem summarizes all the results proved in this
section. 

\begin{thm} 
\label{specialimplpinceau}
Let $X$ be real curve with $s\leq 4$. If $X$ has a
  very special linear system then $X_{\CC}$ is $s$-gonal and
  $X$ has a very special pencil.
\end{thm}

By Theorem \ref{s>=g-4}, the same conclusion can be
drawn if $s\geq g-4$.

An open question is to know if the statement of Theorem
\ref{specialimplpinceau} is correct without any hypothesis on $s$. If the
answer to this question is the affirmative then very special linear series
will only exist on separating real curves.

By Theorem \ref{cliffreelA}, we only have to consider curves with
$2\leq s\leq 4$.

We will use several times the following proposition:
\begin{prop}
\label{criteresimple}
Let $|D|$ be a base point free, simple, and very special linear system
of degree $d$ and index $k$ such that $d\geq g$. Then $$\dim |K-D|\leq
\delta (D)-2k-2.$$
\end{prop}

\begin{proof}
Set $r=\dim |D|$, we have $r=\frac{1}{2}(d-\delta (D))+k+1$ i.e.
$$d=2r+\delta(D)-2k-2.$$
Since $d\geq g$, $2D$ is non-special and therefore
$\dim |2D|=2d-g=4r+2\delta(D)-4k-4-g$. 
Since $|D|$ is simple and base point free, by \cite[Ex. B.6,
Chap. 3]{ACGH} (a consequence of the uniform position lemma) (note
that there is a misprint in the exercise, the correct formula should
be $r(\cal{D}+{\cal E})\geq$ $
r({\cal D})+2r({\cal E})-r({\cal E}-{\cal D})-1$),
we get $\dim |2D|\geq 3r-1$. Therefore 
$$4r+2\delta(D)-4k-4-g\geq 3r-1$$
$$r=\frac{1}{2}(d-\delta (D))+k+1\geq 4k-2\delta(D)+g+3$$
$$d-\delta(D)+2k+2\geq 8k-4\delta(D)+2g+6$$
$$d\geq 6k-3\delta(D)+2g+4.$$
Hence $\deg (K-D)=2g-2-d\leq 3\delta(D)-6k-6$.
By Lemma \ref{residuel}, we obtain finally
$$\dim |K-D|=\frac{1}{2}(\deg (K-D)-\delta (D))+k+1\leq \delta(D)-2k-2.$$
\end{proof}

\begin{prop}
\label{s=2}
Let $X$ be a real curve such that $s=2$. If $X$ has a very special
linear system then $X$ is hyperelliptic and the $g_2^1$ is very
special i.e. $\delta (g_2^1)=2$.
\end{prop}

\begin{proof}
Assume $D$ is a very special divisor of degree $d$ and index $k$. Then
$\dim |D|=r=\frac{1}{2}(d-\delta (D))+k+1\geq \frac{1}{2}d+k$ since
$\delta (D)\leq 2$. By Clifford Theorem we get $k=0$. Since the null
divisor and $K$ are not very special and since we have an equality in
Clifford Theorem, it follows that $X$ is hyperelliptic. By
\cite[Prop. 2.10]{Mo1}, we get $\delta (g_2^1)=2$.
\end{proof}

\begin{prop}
\label{s=3}
 Let $X$ be a real curve such that $s=3$. If $X$ has a very special
linear system then $X_{\CC}$ is a trigonal curve and there exists a
real $g_3^1$ very
special i.e. $\delta (g_3^1)=3$.
\end{prop}

\begin{proof}
Assume $D$ is a very special divisor of degree $d$ and index $k$. Then
$\dim |D|=r=\frac{1}{2}(d-\delta (D))+k+1\geq
\frac{1}{2}d+k-\frac{1}{2}$.
By Clifford Theorem we get $k=0$ i.e. $$r=\frac{1}{2}(d-\delta
(D))+1.$$
If $\delta(D)\leq 2$ we get a contradiction since in this case $X$
would be hyperelliptic (equality in Clifford inequality) and then $s=2$
(Proposition \ref{s=2}).
Therefore $$\delta (D)=3,$$ $d$ is odd and $$r=\frac{1}{2}(d-1).$$
By Lemma \ref{basepointfreepart}, since the index of any special
divisor is null and since the $\delta$ of any very special divisor is
equal to $3$, we can conclude that any very special linear system is
base point free i.e. any very special linear system is primitive.

Assume $D$ is simple and $d\leq g-1$. Castelnuovo's bound gives
$r\leq \frac{1}{3}(d+1)$ \cite{Beau} i.e.
$$\frac{1}{2}(d-1)\leq \frac{1}{3}(d+1)$$
$$d\leq 5.$$
Then $r\leq 2$ and it is impossible by Theorem \ref{dim2}.

Assume $D$ is simple and $d\geq g$. By Proposition \ref{criteresimple}
we get $\dim |K-D|\leq 1$ and it follows that $|K-D|$ is a very
special pencil i.e. a $g_3^1$ with $\delta (g_3^1)=3$.

Assume $D$ is non simple. If $|D|$ is a pencil there is nothing to
do. If $\dim |D|>1$, the existence of a very special $g_3^1$ is given
by Theorem \ref{nonsimple}.

If $X_{\CC}$ is not trigonal then $X_{\CC}$ must be hyperelliptic
since the complex gonality is less than the real gonality and since
$X$ has a special divisor. Since the $g_2^1$ of an hyperelliptic curve
is unique, $X$ must be hyperelliptic. By \cite[Prop. 3.10]{Mo1},
$s=2$, contradiction.
\end{proof}

\begin{lem}
\label{lempours=4}
Let $X$ be a real curve such that $s=4$. If $D$ is a very special
divisor on $X$ then $\ind (D)=0$ and one of the following statements holds:
\begin{description}
\item[({\cal{i}})] $\delta(D)=3$ and $|D|$ is primitive.
\item[({\cal{ii}})] $\delta(D)=4$ and the base part of $|D|$ is empty
  or a real point.
\end{description}
\end{lem}

\begin{proof}
Assume $D$ is a very special divisor of degree $d$ and index $k$. Then
$\dim |D|=r=\frac{1}{2}(d-\delta (D))+k+1\geq
\frac{1}{2}d+k-1$. We get $k=0$
since $X$ can not be hyperelliptic \cite[Prop. 3.10]{Mo1}
(if $k=1$ we have equality in Clifford inequality, if $k>1$ we
contradict Clifford inequality). 
Therefore $$r=\frac{1}{2}(d-\delta (D))+1.$$
We have $\delta (D)=3$ or $\delta (D)=4$ since $X$ is not hyperelliptic
(if $\delta (D)=2$ we have equality in Clifford inequality, if
$\delta(D)\leq 1$ we
contradict Clifford inequality). 

Assume $\delta(D)=3$. By Lemma \ref{basepointfreepart}, 
since the index of any very special
divisor is null and since the $\delta$ of any very special divisor is
equal to $3$ or $4$, we can conclude that $|D|$ is base point
free. Since $\delta(K-D)=3$, $|K-D|$ is also base point free
i.e. $|D|$ is primitive.

Assume $\delta(D)=4$. By Lemma \ref{basepointfreepart}, 
since the index of any very special
divisor is null and since the $\delta$ of any very special divisor is
equal to $3$ or $4$, we conclude that the base part of $|D|$ is empty
or a real point. 
\end{proof}

\begin{thm}
\label{s=4}
Let $X$ be a real curve such that $s=4$. If $X$ has a 
very special linear system then $X_{\CC}$ is a $4$-gonal curve and
there exists a very special $g_4^1$ i.e. $\delta(g_4^1)=4$.
Moreover, if $|D|$ is a very special linear system on $X$ then $|D|$ is
primitive, $\ind (D)=0$ and $\delta(D)=4$.
\end{thm}

\begin{proof}
Assume $D$ is a very special divisor of degree $d$. By Lemma
\ref{lempours=4} we know that $\ind (D)=0$ and that $\delta (D)\geq
3$. 

Suppose first that $\delta (D)=3$. We know that $|D|$ is
primitive (Lemma \ref{lempours=4}) and that $|D|$ and $|K-D|$ are
simple (Theorem \ref{nonsimple}). 
Changing $D$ by $K-D$ if necessary, we may assume that $d\leq g-1$.
We have $$\dim |D|=r=\frac{1}{2}(d-3)+1.$$
By Castelnuovo's bound $$r\leq \frac{1}{3}(d+1)$$ and we get $d\leq 5$
and thus $r\leq 2$. Since $|D|$ is simple, it follows from Theorem
\ref{dim2} that this case is not possible.

Suppose now that $\delta (D)=4$. If $|D|$ is not base point free
then the base divisor is a real point $P$ by Lemma \ref{lempours=4},
but then $D-P$ is a very special divisor (Lemma
\ref{basepointfreepart}) with $\delta(D-P)=3$, we have shown
previously that it is impossible. 
Thus $|D|$ and $|K-D|$ are base point free i.e they are
primitive. 

We have $$\dim |D|=r=\frac{1}{2}(d-4)+1=\frac{1}{2}d-1.$$

Assume $|D|$ is simple and $d\leq g-1$. 
By Castelnuovo's bound $$r\leq \frac{1}{3}(d+1)$$ and we get $d\leq 8$
and $r\leq 3$. By Theorem \ref{dim2} and since $|D|$ is simple we
get $r=3$ and $d=8$. By \cite[Lem. 5.1, Rem. 5.2]{Beau}, $X$ is an
extremal curve in the sense of Castelnuovo i.e. $|D|$ is very ample
and $X\simeq \varphi_{|D|}(X)\subset \PP^3$ is a space curve of
maximal genus. By \cite[p. 118]{ACGH}, $\varphi_{|D|}(X)$ lie on a
unique quadric surface $Q$. By Proposition \ref{quadrique}, $X$
has a very special $g_4^1$.

Assume $|D|$ is simple and $d\geq g$. By Proposition
\ref{criteresimple}, Lemma \ref{residuel} and Theorem \ref{dim2},
$|K-D|$ is a very special $g_4^1$.

Assume $|D|$ is non simple. If $\dim |D|=1$ there is nothing to do. If
$\dim |D|\geq 2$ the existence of a very special $g_4^1$ follows from
Theorem \ref{nonsimple}. 

If the gonality of $X_{\CC}$ is $\leq 3$ then we contradict
\cite[Prop. 3.10]{Mo1} and Theorem \ref{trigonal}.
\end{proof}

\section{Real trigonal curves}

We study the existence of very special linear series on non
hyperelliptic curves with a complex $g_3^1$.

\begin{thm}
\label{trigonal} 
Let $X$ be a real curve such that $X_{\CC}$ is trigonal. Any very
special linear system on $X$ is a very special $g_3^1$ or the residual
of a very special $g_3^1$. In this situation, $s=3$ and $a(X)=0$.
\end{thm}

\begin{proof}
Remark that since $X$ is not hyperelliptic then $s\geq 3$ (Proposition
\ref{s=2}). 

Assume $g\leq 4$. By Theorem \ref{Huisman}, Propositions \ref{M-2} 
and  \ref{s=3}, if $X$ has a very
special system $|D|$ then $g=4$, $s=3$ and $|D|=g_3^1$ or
$|D|=K-g_3^1$. 

Assume $g>4$. Then $X_{\CC}$ has a unique $g_3^1$ and this $g_3^1$
must be real. 
Suppose $D$ is very special of degree $d\leq g-1$, index $k$ and
suppose moreover $|D|$ is primitive. Set $r=\dim |D|$.
Using the fact that the Maroni's
invariant $m$ of $X$ (it is the first scrollar invariant of the
$g_3^1$) is well understood, we have $\frac{g-4}{3}\leq m\leq
\frac{g-2}{3}$, it is proved in \cite[Example 1.2.7]{CKM} that 
$|D|=r.g_3^1$. Since $\delta (g_3^1)=1$ or $3$, we consider these two
cases separatly.\\
$\bullet$ $\delta(g_3^1)=1$: If $r$ is odd then we get $$r=
\frac{1}{2}(3r-1)+k+1$$ i.e. $r<0$, impossible. If $r$ is even 
then we get $$r=
\frac{1}{2}(3r)+k+1$$ i.e. $r<0$, impossible.\\
$\bullet$ $\delta(g_3^1)=3$: If $r$ is odd then we get $$r=
\frac{1}{2}(3r-3)+k+1$$ and therefore $r=1$ i.e. $|D|=g_3^1$. It is easy
to see that the case $r$ is even is not possible.\\

We have proved that any primitive very
special linear system on $X$ is the very special $g_3^1$ or its
residual. Since the index of any primitive very
special linear system on $X$ is null, it follows from Lemma
\ref{basepointfreepart} and Corollary \ref{primitivehull} 
that the index of any very
special linear system on $X$ is null. 
Suppose now $|D|$ is very special but not primitive, let $|D'|$ denote
the primitive hull of the base point free part of $|D|$. By Lemma
\ref{basepointfreepart} and Corollary \ref{primitivehull}, we must
have $\delta(D')<\delta (D)$, impossible since $\delta(D')=s=3$.
\end{proof}

\section{Four-gonal real curves}

In this section, we study the existence of very special linear series
on four-gonal real curves.
We suppose $X$ is a real curve such that $X_{\CC}$ is
$4$-gonal and such that there exist a base point free 
$g_4^1$ on $X$. We do not assume this
$g_4^1$ is unique. In summary, we assume that $X_{\CC}$ and $X$ are
both $4$-gonal.

In this section we will use several times the following Lemma due to
Eisenbud \cite[Lem. 1.8]{CM}. Note that this lemma concerns linear
systems over $\CC$ but the proof works also over $\RR$.

\begin{lem}
\label{Eisenbud}
Let $g_n^t$, $g_m^r$ $(t,r\geq 1)$ be real complete linear systems on a
real curve $X$. If $g_n^t+g_m^r$ has the minimum possible dimension
$t+r$ then there exists a real base point free pencil $g_e^1$ such
that $g_n^t=t.g_e^1$ and $g_m^r=r.g_e^1$.
\end{lem}

\begin{defn}
A linear system $|D|$ is called ``non-trivial'' if it is base point
free and if
$|D|$ and $|K-D|$ have both dimension $\geq 1$.
\end{defn}

\begin{rema} A base point free very special linear system is always
  non-trivial.
\end{rema}

\begin{defn}
We say that a non-trivial linear system $|D|$ of degree $d$ and 
dimension $r$ is\\
$\bullet$ of type 1 (for the $g_4^1$) if it is composed of the $g_4^1$
i.e. $|D|=r.g_4^1$.\\
$\bullet$ of type 2 (for the $g_4^1$) 
if the residual of $|D|$ is composed of the
$g_4^1$ i.e. $|K-D|=r'.g_4^1+F$ with $F$ the base divisor of $|K-D|$
and $r'=\dim |K-D|=g-d+r-1$.
\end{defn}

\begin{prop} (Very special linear systems of type 1)
\label{type1}\\ 
Let $|D|$ be a base point free very special linear system of type 1
for the $g_4^1$. Then $|D|=g_4^1$ and thus the $g_4^1$ must be very special.
In this situation, $s=4$ and $a(X)=0$.
\end{prop}

\begin{proof} 
Let $d$ be the degree of $D$ and let $k$ be the index of $D$.
We have $$\dim |D|=r=\frac{1}{2}(d-\delta (D))+k+1.$$
Since $|D|$ is of type 1, we also have $|D|=r.g_4^1$.

If $r$ is even, then $\delta(D)=0$, impossible since it will
contradict Clifford Theorem.

Assume $r$ is odd, we have $d=4r$ and $\delta(D)=\delta(g_4^1)=0$ or
$2$ or $4$.\\
$\bullet$ If $\delta(g_4^1)=0$ then we contradict Clifford Theorem.\\
$\bullet$ If $\delta(g_4^1)=2$ then $r=\frac{1}{2}(4r-2)+k+1$. It
follows that $k=0$ and we have an equality in Clifford inequality,
impossible since $X_{\CC}$ is $4$-gonal.\\
$\bullet$ If $\delta(g_4^1)=4$ then $s=4$ and the $g_4^1$ is very
special. We have $r=\frac{1}{2}(4r-4)+k+1$ i.e. $r=1-k$ and the proof
is done.
\end{proof}

\begin{prop} (Very special linear systems of type 2)
\label{type2}\\ 
Let $|D|$ be a base point free very special linear system of type 2
for the $g_4^1$. Then $|D|=|K-g_4^1|$ and the $g_4^1$ must be very special.
In this situation $s=4$ and $a(X)=0$.
\end{prop}

\begin{proof}
By Lemma \ref{residuel}, $K-D$ is very special. By Lemma
\ref{basepointfreepart}, the moving part of $|K-D|$ is very special of
type 1. From Proposition \ref{type1}, this moving part is the $g_4^1$
which is very special. Therefore $|K-D|$ is a very special pencil but
then it must be base point free by Theorem \ref{nonsimple}.
Thus $|K-D|=g_4^1$ and the proof is done.
\end{proof}

\begin{prop}
\label{-g41}
Let $|D|$ be a base point free very special linear system such that
$|D|$ is not a pencil (particularly, $|D|$ is not of type 1) and $|D|$ is not
of type 2. Then $$\dim |D-g_4^1|=\dim |D|-2$$ and $|D-g_4^1|$ is base
point free.
\end{prop}

\begin{proof}
Set $r=\dim |D|$ and $d=\deg (D)$. We have $$r-4\leq\dim |D-g_4^1|<r$$ since 
$|D|$ is base point free. Remark that $r\geq 3$.

Assume $\dim |D-g_4^1|=r-1$. Then 
$\dim |D-g_4^1+g_4^1|=r=\dim |D-g_4^1|+\dim g_4^1$.
Let $|E|$ denote the moving part of $|D-g_4^1|$ and let $F$ be the base
divisor of $|D-g_4^1|$. We have 
$$\dim |E|+\dim g_4^1\leq \dim |E+g_4^1|\leq 
\dim |D-g_4^1+g_4^1|=\dim |E|+\dim g_4^1$$ 
hence $$\dim |E|+\dim g_4^1= \dim |E+g_4^1|.$$ 
By Lemma \ref{Eisenbud} we
get $|E|=(r-1).g_4^1$. Hence $|D|=|r.g_4^1+F|$ and $F=0$ since $|D|$ is
base point free. Therefore, $|D|$ is a very special linear system
of type 1 and by Proposition \ref{type1}, we get $r=1$, a contradiction.

Assume $\dim |D-g_4^1|=r-4$. By Riemann-Roch we get $\dim |K-(D-g_4^1)|=\dim
|K-D+g_4^1|= \dim |D-g_4^1|-(d-4)+g-1=r-4-d+4+g-1=r-d+g-1=\dim |K-D|$.
It is impossible since $\dim
|K-D+g_4^1|\geq \dim |K-D|+1$.

Assume $\dim |D-g_4^1|=r-3$. By Riemann-Roch $\dim
|K-(D-g_4^1)|=r-d+g$ i.e. 
$\dim
|K-D+g_4^1|=\dim |K-D|+\dim g_4^1$. 
Let $|E|$ denote the moving part of $|K-D|$ and let $F$ be the base
divisor of $|K-D|$. Set $r'=\dim |K-D|=r-d+g-1$.
By Lemma \ref{residuel} and Lemma \ref{basepointfreepart}, 
$E$ is very special.
We have $\dim |E+g_4^1|\leq \dim
|K-D+g_4^1|=\dim |K-D|+1=\dim |E|+1$. We also have 
$\dim |E+g_4^1|\geq \dim |E|+1$. By Lemma \ref{Eisenbud},
$|E|=r'.g_4^1$ and then $|K-D|=r'.g_4^1+F$ ($F$ is the fix part).
Therefore, $|D|$ is a very special linear system of type 2, a
contradiction with the hypotheses.

We prove now that $|D-g_4^1|$ is base point free. 
Let $|E|$ denote the moving part of $|D-g_4^1|$ and let $F$ be the base
divisor of $|D-g_4^1|$. Let $e$ (resp. $f$) denote the degree of $E$
(resp. $F$). 
We have $\dim |E|=r-2$ and $\dim |E+g_4^1|\geq r-1$.\\
Assume $\dim |E+g_4^1|= r-1$. By Lemma \ref{Eisenbud}, $E=(r-2).g_4^1$
and thus $|D|=|(r-1).g_4^1+F|$. We get 
$$r=\frac{1}{2}(4r-4+f-\delta(D))+k+1$$ i.e.
$$2r=2-f+\delta(D)-2k.$$ If $r$ is odd then $\delta(D)=\delta(F)$ and
we get $$2r=2-f+\delta(F)-2k$$ i.e. $r\leq 1$ since $f\geq \delta(F)$,
contradicting the hypotheses. If $r$ is even then $r\geq 4$ by
Theorem \ref{dim2}. Since $\delta(D)\leq \delta(F)+4$ then
$$2r\leq 6-f+\delta(F)-2k$$ i.e. $r\leq 3$, contradiction.\\
We have proved that $\dim |E+g_4^1|\geq r$ and, since $\dim |E+g_4^1|\leq 
\dim |D|=r$, we get $\dim |E+g_4^1|= r$. Therefore $F$ is contained
in the base divisor of $|D|$ i.e. $F=0$.
\end{proof}

\begin{prop}
\label{+g41}
Let $|D|$ be a primitive very special linear system such that $|K-D|$
is not a pencil and $|D|$ is different from the $g_4^1$.
Then $$\dim |D+g_4^1|=\dim |D|+2$$
and $|D+g_4^1|$ is also primitive.
\end{prop}

\begin{proof}
The linear system $|K-D|$ is base point free, very special and it is
not a pencil. If $|K-D|$ is of type 2 then $|D|$ is of type 1 and we
get $|D|=g_4^1$ by Proposition \ref{type1}, a contradiction with the
hypotheses. We may apply Proposition \ref{-g41} for $|K-D|$ and we get 
$$\dim |K-D-g_4^1|=\dim |K-D|-2$$ and $|K-(D+g_4^1)|$ is base point
free.
By Riemann-Roch, we get $\dim |D+g_4^1|=\dim |D|+2$. To finish the
proof, we remark that $|D+g_4^1|$ is base point free since $|D|$ and
$g_4^1$ are both base point free.
\end{proof}

\begin{thm} (Very special linear systems on a 4-gonal curve with
  $\delta (g_4^1)=0$)
\label{deltag41=0}\\
Let $X$ be a real curve with a fixed $g_4^1$ with $\delta (g_4^1)=0$
and such that $X_{\CC}$ is 4-gonal. Let $|D|$ be a very special linear
system of dimension $r$ on $X$ then\\
$\bullet$ $D$ is primitive.\\
$\bullet$ $r$ is odd, $\ind (D)=0$ and $\delta(D)=s$.\\
$\bullet$ $|D|=|\frac{r-1}{2}.g_4^1+g_s^1|$ with $g_s^1$ a very
special pencil.
\end{thm}

\begin{proof}
We note that a very special linear system on $X$ can not be of type 1 and
can not be of
type 2.
Let $|D|$ be a base point free very special linear system which is not
a pencil. Then $\dim |D-g_4^1|=\dim |D|-2$ by Proposition \ref{-g41}
and it is easy to see that $|D-g_4^1|$ is a base point free very
special linear system of index $\ind (D)$. If $|D-g_4^1|$ is not a
pencil, we continue the same process, and so on, and by Theorem
\ref{dim2} it follows that $D-\frac{r-1}{2}g_4^1$ is a very special
pencil $g_s^1$. We also obtain that $\ind (D)=\ind (g_s^1)=0$ and
$\delta(D)=\delta(g_s^1)= s$.
Since the index of any base point free very special linear system is
null, it follows from Lemma \ref{basepointfreepart} that the index 
of any very special linear system is also 
null. Since the index is always null and the $\delta$ invariant of any
base point free very special linear system is equal to $s$, it follows
from Lemma \ref{basepointfreepart} that the base divisor of any very
special linear system is also
null. 
\end{proof}

\begin{defn}
We say that a non-trivial linear system $|D|$ such that $\dim |D|=r$ 
is of type 3 (for the
$g_4^1$) if $$|D|=|(r-1).g_4^1+F|$$ with $F$ effective. Note then
that $\dim |F|\leq 1$, and for $F\not= 0$ we have 
$\dim (r-1).g_4^1=r-1$.
\end{defn}

\begin{prop}
\label{type3}
Let $|D|$ be a base point free very special linear system of type
3. Then $|D|$ is a very special pencil i.e. a $g_s^1$ with
$\delta(g_s^1)=s$.
\end{prop}

\begin{proof}
Let $d$ be the degree of $D$ and let $k$ be the index of $D$.
We have $$\dim |D|=r=\frac{1}{2}(d-\delta (D))+k+1.$$

Assume first that $F=0$. 
Since $d=4r-4$ we get $$2r=2+\delta(D)-2k.$$ 
If $r$ is odd then $\delta(D)=0$ and we get $r=1-k$, impossible.
If $r$ is even then $r\geq 4$ by  Theorem \ref{dim2}.
Since $\delta(D)\leq 4$ we get $r\leq 3-k$, again impossible.

Assume now that $F\not= 0$ and let $f$ denote its degree. Since $d=4r-4+f$
we get $$2r=2-f+\delta(D)-2k.$$
If $r$ is odd then $\delta (D)=\delta(F)$ and we have
$2r=2-f+\delta(F)-2k$. Since $f\geq \delta(F)$ it follows that $r=1$,
$\delta(F)=f$ and $|D|=|F|$ is a very special pencil.
If $r$ is even then $r\geq 4$ by  Theorem \ref{dim2}. Since
$\delta(D)\leq \delta(F)+4$ we get $2r\leq 6+\delta(F)-f-2k$
i.e. $r\leq 3$, impossible.
\end{proof}

\begin{prop}
\label{nonsimpledeltag41=2}
Let $X$ be a real curve with a fixed $g_4^1$ with $\delta (g_4^1)=2$
and such that $X_{\CC}$ is 4-gonal. If $|D|$ is a base point free
non-simple very
special linear system on $X$ then $|D|$ is a pencil.
\end{prop}

\begin{proof}
Let $|D|$ be a base point free
non-simple very
special linear system on $X$ such that $\dim |D|>1$. By Theorem
\ref{nonsimple}, $\varphi_{|D|}:X\rightarrow X'$ has degree two and
$X'$ is an M-curve of genus $g'=\frac{s}{2}-1$ and the inverse image of
any connected component of $X'(\RR)$ is a disjoint union of two
connected components of $X(\RR)$. By Propositions \ref{s=2} and
\ref{s=3}, we have $s\geq 4$. By Theorem \ref{Huisman} we get $4\leq
s\leq g-1$.
Assume the $g_4^1$ is not induced by $X'$ i.e. $g_4^1$ is not of the
form $\varphi_{|D|}^*(g_2^1)$ for a $g_2^1$ on $X'$. By
\cite[Cor. 2.2.2]{CKM}, we must have $4\geq g-2g'+1$ and $g'\geq 1$. 
Since $s=2g'+2$ and $g\leq 2g'+3$ we obtain $s\geq g-1$.
By Theorem \ref{Huisman} and Proposition \ref{M-2}, we get 
$s=g-1$ and $|D|$ is a pencil, impossible.
Hence $g_4^1=\varphi_{|D|}^*(g_2^1)$ for a $g_2^1$ on $X'$. 
Thus $\delta(g_4^1)\not= 2$, a contradiction.
\end{proof}

\begin{lem}
\label{dim3deltag41=2}
Let $X$ be a real curve with a fixed $g_4^1$ with $\delta (g_4^1)=2$
and such that $X_{\CC}$ is 4-gonal. If $|D|$ is a very
special linear system on $X$ then $$\dim |D|\not=3.$$
\end{lem}

\begin{proof}
Assume $|D|$ is base point free and very special with $\deg(D)=d$ and 
$\dim |D|=3$. From Proposition \ref{dim3}, it follows that $|D|$ is
base point free and that $\delta (D)=s$.
By Proposition \ref{nonsimpledeltag41=2}, $|D|$ is simple.
By Proposition \ref{-g41}, $|D|=|g_4^1+g_s^1|$ with
$\delta(g_s^1)=s-2$. 
But then
$\varphi_{|D|}(X)$ is contained in a quadric surface of $\PP^3$ (see
\cite[Lem. 1.5]{K} for example). By Proposition \ref{quadrique}, $|D|$
cannot be very special, a contradiction.
\end{proof}

\begin{prop}
\label{pinceaudeltag41=2}
Let $X$ be a real curve with a fixed $g_4^1$ with $\delta (g_4^1)=2$
and such that $X_{\CC}$ is 4-gonal. If $X$ has a very special pencil
$g_s^1$ 
then $s=g-1$ and any very special linear system on $X$ is a
pencil.
\end{prop}

\begin{proof}
Let $g_s^1$ be a very special pencil on $X$. 

If $\dim |K-g_s^1|=1$ then we get $s=g-1$ by Riemann-Roch.

For the rest of the proof, we assume $\dim |K-g_s^1|>1$.
We denote by $|D|$ the
base point free
linear system $|g_s^1+g_4^1|$.
By Lemma \ref{+g41}, 
$\dim |g_s^1+g_4^1|=3$ and $|D|=|g_s^1+g_4^1|$ is base point free.

Suppose first that $|D|$ is simple. The curve $X'=\varphi_{|D|}(X)$ is
birational to $X$ and $X'$ is contained in a quadric surface of
$\PP^3$. Thus $X'=\varphi_{|D|}(X)\subset \PP^3$ is a curve of bi-degree
$(a,b)$ on $Q$ and we have
$$a+b=s+4.$$
Arguing as in the proof of Proposition \ref{quadrique}, 
if $a',b',c'$ denote respectively the
number of connected components of type $(1,0)$, $(0,1)$ and $(1,1)$ then
we have $a'+b'= \delta(D)=s-2$, $a'+c'=\delta (g_s^1)=s$ and 
$b'+c'=\delta (g_4^1)=2$. Therefore $a'+b'=a'-b'$ and thus $b'=0$,
$a'=s-2$, $c'=2$. Since the $s-2$ connected components of type $(1,0)$
intersect each connected component of type $(1,1)$, the genus formula
gives $$g\leq 4s -4-s+1-2(s-2)=s+1.$$
By Theorem \ref{cliffreelA}, we get $s=g-1$ and the rest of the proof
in this case follows from Proposition \ref{M-2}.

Suppose now that $|D|$ is not simple. Let $d=s+4$ denote the degree of $D$.
It means that $\varphi_{|D|}$
has some degree $\geq 2$ i.e.
$\varphi=\varphi_{|D|} :X\rightarrow X'$ is a non-trivial covering
map of degree $t\geq 2$ on a real curve $X'$ of genus $g'$. Moreover, 
there exists $D'\in \Div(X')$ of degree $d'=\frac{d}{t}$ such that
$|D'|=g_{d'}^{3}$ and such that $D=\varphi^* (D')$.

Assume $t\geq 3$. Let $Q'+\bar{Q'}$ be a non-real point of
$X'(\RR)$. Let
$D_1=D-\varphi^*(Q'+\bar{Q'})$ and 
denote by $d_1=d-2t$. We may clearly assume $D'-Q'-\bar{Q'}$ effective and 
$\dim |D_1|=1$. Since
$Q'+\bar{Q'}$ 
is non-real,
$\varphi^* (Q'+\bar{Q'})$ is non-real. 
We have  $\delta(D_1)=\delta(D)=s-2$ and 
$D_1$ is clearly a special divisor. We get
$\dim |D_1|=1=\frac{1}{2} (d-\delta(D))-2>\frac{1}{2}
(d_1-\delta(D_1))$, hence $|D_1|$ is a very special pencil, impossible
since $\delta (D_1)\not= s$.

We have $t=2$ and thus $d'=\frac{s}{2}+2$. 
Let $C_1,\ldots,C_{s-2}$ (resp. $C_{s-1},\, C_s$) denote the connected
components of $X(\RR)$ on which the degree of the restriction of $D$
is odd (resp. even).
The image of a connected component of $X(\RR)$ is either a connected
component of $X'(\RR)$ or a closed and bounded interval of a connected
component of $X'(\RR)$. Since $D$ is a union of fibers of
$\varphi_{|D|}$ we get:\\
$\bullet$ for $i=1,\ldots,s-2$, $\varphi(C_i)$ is a connected component of
$X'(\RR)$.\\
$\bullet$  for $i=1,\ldots,s-2$ and for $j=s-1,s$, $\varphi (C_i)\cap
\varphi(C_j)=\emptyset$.\\ 
$\bullet $ for $i=1,\ldots,s-2$, $\varphi^{-1}(\varphi (C_i))$ is
either $C_i$ or $C_i\cup C_{i'}$ for $i'\in \{1,\ldots,s-2\}$ distinct
from $i$.\\
Let $s'$ denote the number of connected components of $X'(\RR)$.
From above remarks, we get
$$s'\geq \frac{s-2}{2}+1=\frac{s}{2}$$ and $$\delta (D')\geq
\frac{s-2}{2}.$$ 
Assume $D'$ is special. We have 
$$\dim |D'|=3=\dim |D|=\frac{1}{2}(d-\delta(D))=d'-\frac{s-2}{2}\geq 
d'-\delta (D')\geq \frac{1}{2}(d'-\delta (D')).$$
Since $\dim |D'|$ is odd, it follows from the above inequalities that $D'$ is
very special. By Proposition \ref{dim3}, $d'=s'+4$ and $\delta(D')=s'$
and thus $\dim |D'|=3=d'-\frac{s-2}{2}\geq 
d'-\delta (D')\geq 4$, a contradiction.\\
Since $D$ is non-special, Riemann-Roch gives $dim
|D'|=3=d'-g'=\frac{s}{2}+2-g'$ i.e.
$$g'=\frac{s}{2}-1.$$
Since $s'\geq \frac{s}{2}$, we get $s'=g'+1=\frac{s}{2}$ by Harnack inequality
i.e. $X'$ is an M-curve. Moreover, from above remarks, it follows that
there exist $g'$ connected components of $X'(\RR)$ such that the
inverse image by $\varphi$ of each of these components 
is a union of two connected components of
$X(\RR)$ among $C_1,\ldots,C_{s-2}$; the connected component of
$X'(\RR)$ that remains contains the image of $C_{s-1}$ and $C_s$.

We know that $4\leq s\leq g-1$ by Theorem \ref{cliffreelA}. 
If the $g_4^1$ is not induced by $\varphi$ then $4\geq g-2g'+1$
(\cite[Cor. 2.2.2]{CKM}) and we find again $s=g-1$.
Now assume $g_4^1=\varphi^*(h_2^1)$ for a $h_2^1$ on $X'$. 
We must consider two cases.\\
$\bullet$ $g'=1$: We have $s'=2$ and $s=4$. The very special pencil
$g_s^1$ is clearly not induced by $\varphi$ since
$\delta(g_s^1)=4$). By \cite[Cor. 2.2.2]{CKM}, we obtain $s=g-1$.\\
$\bullet$ $g'>1$: It follows in that case that $X'$ is an
hyperelliptic M-curve. Therefore there exists $P'\in X'(\RR)$ such
that $h_2^1=|2P'|$ but then $\delta(g_4^1)=0$, a contradiction.
\end{proof}

\begin{thm} (Very special linear systems on a 4-gonal curve with
  $\delta (g_4^1)=2$)
\label{deltag41=2}\\
Let $X$ be a real curve with a fixed $g_4^1$ with $\delta (g_4^1)=2$
and such that $X_{\CC}$ is 4-gonal. Let $|D|$ be a very special linear
system on $X$ then $|D|$ is a pencil and $s=g-1$.
\end{thm}

\begin{proof}
According to Proposition \ref{pinceaudeltag41=2}, it is sufficient to
show that $X$ must have a very special pencil in the case $X$ has a very special
linear system.

We note that a very special linear system on $X$ cannot be of type 1 and
cannot be of
type 2.

Let $|D|$ be a base point free very special linear system of degree
$d$ such that
$\ind (D)=k\geq 1$. We have $$\dim |D|=r=\frac{1}{2}(d-\delta(D))+k+1.$$
By Lemma \ref{dim3deltag41=2}, $r\geq 4$. From Proposition \ref{-g41},
it follows that $$\dim |D-g_4^1|=r-2$$ and $|D-g_4^1|$
is base point free. 
Since $\delta(D-g_4^1)-2\leq \delta(D)\leq \delta(D-g_4^1)+2$, we get 
$$\dim |D-g_4^1|=\frac{1}{2}((d-4)-\delta(D))+k+1\geq
\frac{1}{2}(\deg(D-g_4^1)-\delta(D-g_4^1))+(k-1)+1$$ 
and it follows that $|D-g_4^1|$ is also base point free and very
special. Since $r\geq 4$, $|D-g_4^1|$ is not a pencil and, according
to Proposition \ref{-g41}, we obtain 
$$\dim |D-2g_4^1|=\dim |D|-4$$ and $|D-2g_4^1|$ is base point free. It
is easy to see that $|D-2g_4^1|$ is very special of index
$k=\ind(D)$. If $\dim |D-2g_4^1|\geq 4$ 
then repeating the same process we obtain finally a base point free
very special linear system of dimension $\leq 3$ and index $k\geq 1$,
impossible by the Theorems \ref{nonsimple}, \ref{dim2} and Proposition
\ref{dim3}. Since the index of the base point free part is greater or
equal 
than the index of a very special linear system (Lemma
\ref{basepointfreepart}), it follows that the index of any very
special linear system is null. 

Let $|D|$ be a base point free very special linear system of degree
$d$ such that $|D|$ is not a pencil. We have 
$$\dim |D|=r=\frac{1}{2}(d-\delta(D))+1.$$
By Lemma \ref{dim3deltag41=2}, $r\geq 4$
and it follows from
Proposition \ref{-g41} that $$\dim |D-g_4^1|=r-2$$ and $|D-g_4^1|$
is base point free. We can compare $\delta(D)$ and $\delta(D-g_4^1)$,
we have 3 possibilities.\\

$\bullet$ Case $\delta(D-g_4^1)=\delta(D)+2$: We have 
$$\dim |D-g_4^1|=\frac{1}{2}(deg(D-g_4^1)-\delta(D-g_4^1))+2$$
and $|D-g_4^1|$ is very special of index 1, impossible by an above
conclusion.\\

$\bullet$ Case $\delta(D-g_4^1)=\delta(D)$: We have 
$$\dim |D-g_4^1|=\frac{1}{2}(deg(D-g_4^1)-\delta(D-g_4^1))+1$$
and then $|D-g_4^1|$ is a base point free very special linear system.
Remark that $\delta(D)=\delta(D-g_4^1)<s$ since $\delta(g_4^1)=2$.
If $|D-g_4^1|$ is not a pencil the we may repeat the same process
since $\delta(D-2g_4^1)=\delta(D)=\delta(D-g_4^1)$, and we finally
get a very special pencil (we use Theorem \ref{dim2} and Lemma
\ref{dim3deltag41=2} to exclude the case the linear system we
obtain has dimension 2 or 3) with $\delta$
invariant $< s$, impossible by Theorem \ref{nonsimple}.\\

We have proved that 
$$\delta(D-g_4^1)=\delta(D)-2$$ and we recall that $|D-g_4^1|$ is base
point free by Proposition \ref{-g41}. We also remark that $\dim
|D-g_4^1|\geq 2$. Since $|D-g_4^1|$ is base
point free we have 
$$r-6\leq\dim |D-2g_4^1|\leq r-3.$$

Suppose $\dim |D-2g_4^1|=\dim |D|-6\geq 0$. Then 
$$\dim |K-(D-g_4^1)+g_4^1|=\dim |K-(D-2g_4^1)=\dim |K-(D-g_4^1)|$$ 
by Riemann-Roch,
impossible since $h^0(K-(D-g_4^1))=h^1(D-g_4^1)>0$. Thus 
$\dim |K-(D-g_4^1)+g_4^1|\geq \dim |K-(D-g_4^1)|+1$.\\

Suppose $\dim |D-2g_4^1|=\dim |D|-3\geq 0$. Then 
$$\dim |(D-2g_4^1)+g_4^1|=\dim |D-2g_4^1|+\dim |g_4^1|.$$
By Lemma \ref{dim3deltag41=2}, we get $\dim |D-2g_4^1|>0$. Let $|E|$
(resp. $F$) denote
the base point free part (resp. the base part) of $|D-2g_4^1|$ then 
$$\dim|E|+\dim |g_4^1|=\dim |D-2g_4^1|+\dim |g_4^1|=\dim
|(D-2g_4^1)+g_4^1|\geq \dim |E+g_4^1|\geq \dim|E|+\dim |g_4^1|$$ 
i.e. $$\dim |E+g_4^1|=\dim|E|+\dim |g_4^1|.$$
By Lemma \ref{Eisenbud}, $|E|=(r-3).g_4^1$. It follows that
$|D|=|(r-1).g_4^1+F|$ i.e. $|D|$ is a very special linear system of
type 3; Proposition \ref{type3} gives a contradiction since $r\geq
4$.\\

Suppose $\dim |D-2g_4^1|=\dim |D|-5\geq 0$. By Riemann-Roch 
$$\dim |K-(D-g_4^1)+g_4^1|=\dim |K-(D-g_4^1)|+\dim |g_4^1|.$$
Assume $\dim |K-(D-g_4^1)|>0$ and denote by $|E'|$
(resp. $F'$) 
the base point free part (resp. the base part) of $|K-(D-g_4^1)|$.
We have $$\dim|E'|+\dim |g_4^1|=\dim |K-(D-g_4^1)|+\dim |g_4^1|$$ $$=\dim
|K-(D-g_4^1)+g_4^1|\geq \dim |E'+g_4^1|\geq \dim|E'|+\dim |g_4^1|$$ 
i.e. $$\dim |E'+g_4^1|=\dim|E'|+\dim |g_4^1|.$$
By Riemann-Roch, $\dim |E'|=\dim
|K-(D-g_4^1)|=r-2-(d-4)+g-1=r-d+g+1$. 
By Lemma \ref{Eisenbud}, $|E'|=|(r-d+g+1).g_4^1|$. Hence
$|K-(D-g_4^1)|=|K-D+g_4^1|=|(r-d+g+1).g_4^1|+F'$ and then
$|K-D|=|(r-d+g).g_4^1+F'|$. It follows that $\dim |K-D|\geq r-d+g$,
impossible by Riemann-Roch. Hence $\dim |K-(D-g_4^1)|=0$ i.e. 
$\dim |K-D+g_4^1|=0$, again impossible since $D$ is special.\\

According to above, we have $$\dim |D-2g_4^1|=\dim |D|-4.$$ Since $r\geq 4$ then
$h^0(D-2g_4^1)>0$ and $D-2g_4^1$ is special. Since
$\delta(D)=\delta(D-2g_4^1)$ it is easy to see that $|D-2g_4^1|$ is
very special. Let $|D_1|$ denote the base point free part of
$|D-2g_4^1|$. By Lemma \ref{basepointfreepart} and an above remark,
$|D_1|$ is very special of index null. If $\dim |D-2g_4^1|\leq 3$ then 
$|D_1|$ must be a very special pencil $g_s^1$ and $|D-2g_4^1|=|D_1|$
since a very special pencil is primitive.
If $\dim |D-2g_4^1|\geq 4$, we do the same process with $|D_1|$
replacing $|D|$. 
We have proved the existence of a very special pencil on $X$, the rest of
the proof follows now from Proposition \ref{pinceaudeltag41=2}.
\end{proof}

In the following we describe all the very special linear sytems on a
real curve with four real connected components (see Theorem
\ref{specialimplpinceau}). 
\begin{thm} (Very special linear systems on a 4-gonal curve with
  $\delta (g_4^1)=4$)
\label{deltag41=4}\\
Let $X$ be a real curve with a fixed $g_4^1$ with $\delta (g_4^1)=4$
and such that $X_{\CC}$ is 4-gonal. Let $|D|$ be a very special linear
system of degree $d$ and dimension $r$ on $X$. Then $|D|$
is primitive, $r$ is odd, $\ind (D)=0$, $\delta(D)=s=4$.
Moreover, we are in one of the following cases:\\
$\bullet$ $|D|$ is simple and $d\leq g-1$: then $r=3$,
$|D|=|g_4^1 +h_4^1|$ 
with
$h_4^1$ another pencil such that $\delta(h_4^1)=0$,
$\varphi_{|D|}(X)$ is a smooth curve of bidegree $(4,4)$ on a quadric surface
$Q$ of $\PP^3$.\\
$\bullet$ $|D|$ is simple and $d\geq g$: then $|D|=|K-h_4^1|$ with 
$h_4^1$ a very special pencil.\\
$\bullet$ $|D|$ is a very special pencil $h_4^1$.\\
$\bullet$ $|D|$ is non simple and is not a pencil:
then $X$ is a bi-elliptic curve and $$|D|=|\frac{r-1}{2}g_4^1+h_4^1|$$
with $h_4^1$ a pencil such that $\delta(h_4^1)=4$ (i.e. very
special) if $r=1\mod\,4$ and $\delta(h_4^1)=0$ if $r=3\mod\,4$.
\end{thm}

\begin{proof}
Let $|D|$ be a very special linear system of degree $d$ and dimension
$r$. By Theorem \ref{s=4}, $|D|$ is
primitive, $\ind (D)=0$ and $\delta(D)=s=4$.
We have $$\dim |D|=r=\frac{1}{2}(d-4)+1=\frac{d}{2}-1.$$

Assume $|D|$ is simple and $d\leq g-1$. 
By Castelnuovo's bound $$r\leq \frac{1}{3}(d+1)$$ and we get $d\leq 8$
and $r\leq 3$. By Theorem \ref{dim2} and since $|D|$ is simple we
get $r=3$ and $d=8$. By Proposition \ref{-g41}, $|D-g_4^1|$ is a base
point free pencil $h_4^1$ with $\delta(h_4^1)=0$. It follows that 
$\varphi_{|D|}(X)$ is a curve of bidegree $(4,4)$ on a quadric surface
$Q$ of $\PP^3$. By Proposition \ref{quadrique}, $\varphi_{|D|}(X)$ is
smooth i.e. $|D|$ is very ample. By \cite[p. 118]{ACGH}, the quadric
containing $\varphi_{|D|}(X)$ is unique.

Assume $|D|$ is simple and $d\geq g$. By Proposition
\ref{criteresimple}, Lemma \ref{residuel} and Theorem \ref{dim2},
$|K-D|$ is a very special pencil $h_4^1$ i.e. $|D|$ is of type 2 for
that $h_4^1$.

Assume $|D|$ is non simple and is not a pencil.
By Theorem
\ref{nonsimple} $r$ is odd, $\varphi_{|D|}:X\rightarrow X'$ has degree two and
$X'$ is an elliptic curve with two real connected components 
and the inverse image of
any connected component of $X'(\RR)$ is a disjoint union of two
connected components of $X(\RR)$. From Theorem \cite[Thm. 4.1]{Mo2}
and using Proposition \ref{-g41} and \cite[Example 1.13]{CM}, we see
that $$|D|=|\frac{r-1}{2}g_4^1+h_4^1|$$
with $h_4^1$ a pencil such that $\delta(h_4^1)=4$ (i.e. very
special) if $r=1\mod\,4$ and $\delta(h_4^1)=0$ if $r=3\mod\,4$.
\end{proof}

From \cite[Prop. 2.10]{Mo1}, Theorems \ref{trigonal},
\ref{deltag41=0}, \ref{deltag41=2} and \ref{deltag41=4}, we get
Theorem \ref{introthm2} stated in the introduction.
\begin{thm} 
\label{introthm2new}
Let $X$ be real curve such that $X$ and $X_{\CC}$ are both $n$-gonal
with $2\leq n\leq 4$. If $X$ has a
very special linear system then 
$X$ has a very special pencil and $X$ is a separating curve
i.e. $a(X)=0$.
Moreover, if $|D|$ is a very special linear system then $\ind(D)=0$,
$\delta(D)=s$ and $|D|$ is primitive.
\end{thm} 

\section{Clifford type inequality for very special linear systems}

Using the results of the previous sections, we will improve the
inequalities of Theorem \ref{cliffreelA} and Theorem \ref{cliffreelB}.

\begin{thm}
\label{cliffreelC}
Let $X$ be a real curve such that $X$ is not hyperelliptic and $X$ is
not trigonal, i.e. the real gonality of $X$ is $\geq 4$. Let $D$ be a
very special divisor of degree $d$ and index $k$ then 
$$\dim |D|\leq \frac{d}{2}-\frac{s}{4}.$$
\end{thm}

\begin{proof}
Let $|D|$ be a very special linear system of degree $d$, index $k$ and
dimension $r$.
Before proving the inequality stated in the Theorem, we will prove the
following inequality $$\dim |D|\leq \frac{1}{2}(d-\beta(D))-k-1.$$

Assume $|D|$ is a pencil. By Proposition \ref{veryspecialpencil}, we have   
$\delta(D)=s$, $d=s$ and $k=0$. 
According to Propositions \ref{s=2} and \ref{s=3}, we have $s\geq 4$
and thus $$\dim |D|=1\leq
\frac{1}{2}(d-\beta(D))-k-1=\frac{1}{2}s-1.$$
In the following of the proof, we assume $|D|$ is not a pencil.
By Lemma \cite[Lem. 2.5]{Mo1} and Lemma \ref{residuel}, we may assume
$|D|$ is base point free and $d\leq g-1$.

We have
\begin{equation}
\label{equ1}
r=\frac{1}{2} (d-\delta (D))+k+1
\end{equation}
and suppose 
\begin{equation}
\label{equ2}
r>\frac{1}{2} (d-\beta (D))-k-1
\end{equation}
By Proposition \ref{condition2}, we get
\begin{equation}
\label{equ3}
d+\delta(D)\geq 2s+2k+4
\end{equation}
From (\ref{equ2}), (\ref{equ3}) and since $\beta(D)=s-\delta(D)$, we
get 
\begin{equation}
\label{equ4}
r\geq \frac{s}{2}+\frac{3}{2}
\end{equation}
By (\ref{equ1}) and (\ref{equ2}), we obtain 
\begin{equation}
\label{equ5}
2r\geq d-\frac{s}{2}+\frac{1}{2}
\end{equation}
Using (\ref{equ4}) and (\ref{equ5}), it follows that
\begin{equation}
\label{equ6}
3r\geq d+2
\end{equation}
If $|D|$ is simple, there is a contradiction with Castelnuovo's bound 
$3r\leq d+1$.

Therefore $|D|$ is non simple and we know that $\delta(D)=s$ and $k=0$ in
that case by Theorem \ref{nonsimple}. From By (\ref{equ1}) and
(\ref{equ2}), we get
$$\frac{1}{2} (d-s)+1> \frac{1}{2}d-1$$
i.e. $$s\leq 3.$$
The case $s=1$ is not possible by Theorem \ref{Huisman}. 
The case $s=2$ (resp $s=3$) is impossible by Proposition \ref{s=2}
(resp. \ref{s=3}) and given the hypotheses.

Set $A=\frac{1}{2}(d-\beta(D))-k-1$ and
$B=\frac{d}{2}-\frac{s}{4}$. Then $$r+A=2B.$$
Therefore, since we have proved that $r\leq A$ then 
$$r\leq B\leq A$$ and the proof is done.
\end{proof}

We are interested by the case when we have an equality in the
inequality given in the previous theorem.

We introduce a new invariant of very special linear systems.
\begin{defn}
Let $D$ be a very special divisor. The rational number $l\in\QQ$ with
$2l\in\ZZ$ such that 
$$\dim |D|=\frac{1}{2}(\deg(D)-\beta(D))-l$$ is called the coindex of
$D$ (or $|D|$) and is denoted by $\coind(D)$.
\end{defn}

\begin{lem} \cite[Lem. 3.6]{Mo1}
\label{residuel2}
Let $D$ be a very special divisor then 
$$\coind(D)=\coind(K-D).$$
\end{lem}

We reformulate Theorems \ref{cliffreelA} and \ref{siegalite} using the
notion of coindex.
\begin{thm} \cite[Thm. 3.8, Thm. 3.18]{Mo1}
\label{coind1}
Let $D$ be a very special divisor then
$$\coind(D)\geq \ind(D).$$
If there is an equality in the previous inequality then $X$ is
hyperelliptic with a very special $g_2^1$ and $|D|=r.g_2^1$ with
$r=\dim |D|$ odd.
\end{thm}

We give a consequence of the proof of Theorem \ref{cliffreelC}.
\begin{cor}
\label{coind2}
Let $X$ be a real curve such that $X$ is not hyperelliptic and $X$ is
not trigonal, i.e. the real gonality of $X$ is $\geq 4$. Let $D$ be a
very special divisor then 
$$\coind(D)\geq \ind(D)+1.$$
\end{cor}

\begin{lem} 
\label{basepointfreepart2}
Let $D$ be an effective divisor. Let $F$ be the base divisor of $|D|$.
If $D$ is very special then 
the base point free part $|E|=|D-F|$ of $|D|$
is also very special and 
$$\coind(E)\leq \coind(D).$$
Moreover $\coind(E)=\coind(D-F)= \coind(D)$ if and only if
$F=\sum_i P_i$ with
the $P_i$ some real points among the $\beta (D)$ real connected
components on which the degree of the restriction of $D$ is even, such
that no two of them belong to the same real connected component.
\end{lem}

\begin{proof}
Set $d=\deg (D)$ and $l=\coind (D)$.

Assume a non-real point $Q+\bar{Q}$ is contained in the base divisor
of $|D|$. Then $\dim |D|=\dim
|D-Q-\bar{Q}|=\frac{1}{2}(d-\beta(D))-l=\frac{1}{2}((d-2)-\beta(D))-(l-1)$ 
and $\coind(D-Q-\bar{Q})=\coind (D)-1$. 

Assume two real points $P,P'$ belonging to the same real connected
component, are contained in the base divisor
of $|D|$, then as before, $\coind(D-P-P')=\coind (D)-1$.

Assume a real point $P$ belonging to a connected component on which
the degree of the restriction of $D$ is even, is a base point of
$|D|$. 
Then $\dim |D|=\dim
|D-P|=\frac{1}{2}(d-\beta(D))-l=\frac{1}{2}((d-1)-(\beta(D)-1))-l
=\frac{1}{2}(\deg(D-P)-\beta(D-P))-l$ and $\coind(D-P)=\coind (D)$.

Assume a real point $P$ belonging to a connected component on which
the degree of the restriction of $D$ is odd, is a base point
of $|D|$. Then $\dim |D|=\dim
|D-P|=\frac{1}{2}(d-\beta(D))-l=\frac{1}{2}((d-1)-(\beta(D)+1))-l+1
=\frac{1}{2}(\deg(D-P)-\beta(D-P))-(l-1)$ and $\coind(D-P)=\coind (D)-1$.
\end{proof}

\begin{lem} (\cite[Lem. 3.1]{ELMS})
\label{accola}
Let $D$ and $E$ be divisors of degree $d$ and $e$ on a curve $X$ of
genus $g$ and suppose that $|E|$ is base point free. Then 
$$h^0 (D)-h^0 (D-E)\leq \frac{e}{2}$$
if $2D-E$ is special.
\end{lem}

The previous lemma applies in case $D$ is semi-canonical i.e. $2D=K$.

\begin{lem} (\cite{Ac}, \cite{Co-Ma} p. 200 and \cite{ACGH} p. 122)
\label{extremale} 
Let $X$ be an extremal curve (it means the genus of the curve is
maximal i.e. the genus equals the Castelnuovo's bound) 
of degree $d>2r$ in $\PP_{\RR}^r$ $(r\geq
3)$. Then one of the followings holds:
\begin{description}
\item[({\cal{i}})] $X$ lies on a rational normal scroll $Y$ in
  $\PP_{\RR}^r$ ($Y$ is real, see \cite{ACGH} p. 120). Write
  $d=m(r-1)+1+\varepsilon$ where
  $m=[\frac{d-1}{r-1} ]$ and
  $\varepsilon\in \{0,1,2,\ldots ,r-2\}$. The curve $X_{\CC}$ has only
  finitely many
  base point free pencils of degree $m+1$ (in fact, only $1$ for
  $r>3$, and $1$ or $2$ if $r=3$). These pencils are swept out by the
  rulings of $Y_{\CC}$. Moreover $X_{\CC}$ has no $g_m^1$.
\item[({\cal{ii}})] $X$ is the image of a smooth plane curve $X'$ of
  degree $\frac{d}{2}$ under the Veronese map
  $\PP_{\RR}^2\rightarrow\PP_{\RR}^5$.
\end{description}
\end{lem}

\begin{prop}
\label{egalitecliffordreelC}
Let $X$ be a real curve such that $X$ is not hyperelliptic and $X$ is
not trigonal, i.e. the real gonality of $X$ is $\geq 4$. Let $D$ be a
very special divisor of degree $d$ and index $k$ such that
$$\dim |D|= \frac{d}{2}-\frac{s}{4}$$ or equivalently such that 
$$\coind(D)=\ind(D)+1=k+1.$$
Then $X_{\CC}$ is 4-gonal and
$X$ has a very special pencil $g_4^1$. Moreover $s=4$, $a(X)=0$,
$\delta(D)=s$, $k=0$ and $|D|$ is one of the linear systems listed in
Theorem \ref{deltag41=4}.
\end{prop}

\begin{proof}
Looking at the proof of Theorem \ref{cliffreelC}, 
we see that the equality $\dim |D|= \frac{d}{2}-\frac{s}{4}$ is equivalent to
the other equality $\coind(D)=\ind(D)+1.$. Thus
$$\dim |D|=r=\frac{1}{2} (d-\delta (D))+k+1=\frac{1}{2} (d-\beta
(D))-k-1.$$ 
We claim $|D|$ is base point free.
By the Lemmas \ref{basepointfreepart}, \ref{basepointfreepart2} and
Corollary \ref{coind2},
if $|E|$ denote the base point free part 
of $|D|$ then we must have $\ind (E)=\ind(D)$ and
$\coind(E)=\coind(D)$ since $\coind(E)\geq\ind(E)+1.$ It follows now from the
lemmas \ref{basepointfreepart} and \ref{basepointfreepart2}
that $|E|=|D|$ since they have the same index and coindex.
By Lemma \ref{residuel2}, $|D|$ is primitive.
By the Lemmas \ref{residuel} and \ref{residuel2}, we may assume $d\leq
g-1$.

We assume first that $|D|$ is non-simple. By Theorem
\ref{nonsimple}, $k=0$ and $\delta(D)=s$. We get
$$\frac{1}{2} (d-s)+1=\frac{d}{2}-\frac{s}{4}$$ i.e.
$$s=4.$$
By Theorem \ref{s=4}, $X$ has a very special $g_4^1$ and $X_{\CC}$ is
4-gonal. We use Theorem \ref{deltag41=4}
to finish the proof in this case.

We assume now that $|D|$ is simple. We have
\begin{equation}
\label{equ7}
r=\frac{1}{2} (d-\delta (D))+k+1
\end{equation}
\begin{equation}
\label{equ8}
r=\frac{1}{2} (d-\beta (D))-k-1
\end{equation}
and 
\begin{equation}
\label{equ9}
r=\frac{d}{2}-\frac{s}{4}
\end{equation}
By Proposition \ref{condition2}, we get
$d+\delta(D)\geq 2s+2k+4$
and we claim that here it is an equality:\\
If $d+\delta(D)\geq 2s+2k+6$ then using (\ref{equ8}) we get 
$$r\geq \frac{s}{2}+2.$$ 
Using now (\ref{equ9}) we have $s=2d-4r$ and replacing in the previous
inequality we get $$3r\geq d+2$$
and this contradicts Castelnuovo's bound.
Therefore, we have
\begin{equation}
\label{equ10}
d+\delta(D)= 2s+2k+4
\end{equation}
From (\ref{equ9}), (\ref{equ10}) and (\ref{equ8}),
it follows that $$r= \frac{s}{2}+1$$ and that
$$3r=d+1$$ i.e. $\varphi_{|D|}(X)$ 
is an extremal curve in the sense of Castelnuovo.
By \cite[Lem. 2.9]{Ac2}, $D$ is
semi-canonical i.e. $|2D|=|K|$. 
We denote by $Y$ the curve $\varphi_{|D|}(X)$.
We have $m=[\frac{d-1}{r-1}]=[\frac{3r-2}{r-1}]=3$ since $r\geq
3$. By Lemma \ref{extremale}, we have to consider the following
cases.\\
\\
{\sc Case 1} $r=5$ and $X$ is a smooth plane curve:\\
By Lemma \ref{extremale}, $Y$ is the image of a smooth plane curve of
degree $7$ under the Veronese embedding $\PP^2\hookrightarrow \PP^5$. The
curve $X$ has a unique very ample $g_7^2$ which calculate the Clifford
index of $X_{\CC}$. Since $D$ is semi-canonical,
by Lemma \ref{accola}, the linear system $|D-g_7^2|$ of degree $7$ 
has dimension $\geq 2$. Since the Clifford index of $X$ is $3$, we
have $\dim |D-g_7^2|=2$. It follows that $|D|=2g_7^2$ and
$\delta(D)=0$, impossible.\\
\\
{\sc Case 2} $r\geq 4$ and $X$ is not a smooth plane curve:\\
By Lemma \ref{extremale}, $X$ has a $g_4^1$ and $X_{\CC}$ is $4$-gonal.
From the Theorem \ref{deltag41=0}, \ref{deltag41=2} and
\ref{deltag41=4}, it follows that $k=0$ and $\delta(D)=s$. 
By (\ref{equ10}) and (\ref{equ7}), $d=s+4$ and $r=3$, impossible.\\
\\
{\sc Case 3} $r=3$:\\
By Proposition \ref{dim3}, $d=s+4$, $k=0$ and 
$\delta (D)=s$. Since in this case the rational scroll is a quadric
surface, the existence of the very special $g_4^1$ follows from
Proposition \ref{quadrique}.
\end{proof}

We summarize the results of this section in the following theorem.
\begin{thm}
\label{resume}
Let $|D|$ be a very special linear system of degree $d$ on a real curve $X$.
\begin{description}
\item[({\cal{i}})] We have $$\dim |D|\leq
  \frac{1}{2}(d-\frac{s-2}{2}),$$
with equality i.e. $$\coind (D)=\ind (D)$$ if and only if $X$ is
hyperelliptic, the $g_2^1$ is very special and $s=2$.
\item[({\cal{ii}})] Assume $X$ is not hyperelliptic. We have $$\dim |D|\leq
  \frac{1}{2}(d-\frac{s-1}{2}),$$
with equality i.e. $$\coind (D)=\ind (D)+\frac{1}{2}$$ if and only if $X$ is
trigonal, a $g_3^1$ is very special and $s=3$.
\item[({\cal{iii}})] Assume $X$ is not hyperelliptic and not
  trigonal. 
  We have $$\dim |D|\leq
  \frac{1}{2}(d-\frac{s}{2}),$$
with equality i.e. $$\coind (D)=\ind (D)+1$$ if and only if $X$ is
4-gonal, a $g_4^1$ is very special and $s=4$.
\item[({\cal{iv}})] Assume $X$ has gonality $\geq 5$. We have $$\dim |D|\leq
  \frac{1}{2}(d-\frac{s+1}{2}).$$
\end{description}
\end{thm}

\begin{proof}
The proof of the theorem follows from the results of the paper except maybe
the part concerning equality in ({\cal{ii}}).

If $X$ is trigonal with a very special $g_3^1$ then $s=3$ and 
$\dim g_3^1=\frac{1}{2}(deg(g_3^1)-\frac{s-1}{2})$.

Assume $X$ is not hyperelliptic and suppose there is a very special linear
system $|D|$ of degree $d$ such that $$\dim |D|=
  \frac{1}{2}(d-\frac{s-1}{2}).$$
By Theorems \ref{cliffreelC} and \ref{siegalite}, $X$ is trigonal. By Theorem
\ref{trigonal}, a $g_3^1$ is very special.
\end{proof}

\end{document}